\documentclass{article}
\usepackage{amssymb,amsmath,amsfonts, amsthm, amscd}
\usepackage{arydshln}
\usepackage{float}
\usepackage{graphicx}
\usepackage{caption}
\usepackage{rotating}
\usepackage{pdflscape}
\usepackage{tikz}
\usepackage[utf8]{inputenc}
\usepackage{booktabs}
\usepackage{lipsum}
\usepackage{multirow}
\usepackage{enumitem}
\usepackage{subcaption}
\usepackage{multirow}
\usepackage{epsfig,amsthm,amsmath,amsfonts,amssymb}
\usepackage{fancyhdr,multicol,color,xcolor}
\usepackage[english]{babel}
\usepackage[utf8]{inputenc}
\usepackage{algorithm}
\usepackage[noend]{algpseudocode}
\numberwithin{equation}{section} \numberwithin{equation}{section}

\newcommand{\RNum}[1]{\uppercase\expandafter{\romannumeral #1\relax}}
\label{sec.sub.gdn}

\newtheorem{definition}{Definition}
\newtheorem{corollary}{Corollary}
\newtheorem{ex}{Example}
\newtheorem{theorem}{Theorem}

\newtheorem{remark}{Remark}

\numberwithin{equation}{section}

\begin{document}

	\title{\bf  
		Analysis of  Linear Systems over Idempotent Semifields}
	\author{{Fateme Olia\textsuperscript{$a$}, Shaban Ghalandarzadeh\textsuperscript{$a$}, Amirhossein Amiraslani\textsuperscript{$b,a$}, Sedighe Jamshidvand\textsuperscript{$a$}}\\
		{\em \small   $^a$ Faculty of Mathematics, K. N. Toosi University of
			Technology, Tehran, Iran} \\
		{\em \small   $^b$ STEM Department, The University of Hawaii-Maui College, Kahului, Hawaii, USA }}
	\maketitle

\begin{abstract}In this paper, we present and analyze methods for solving a system of linear equations over idempotent semifields. The first method is based on the pseudo-inverse of the system matrix. We then present a specific version of Cramer's rule which is also related to the pseudo-inverse of the system matrix. In these two methods, the constant vector plays an implicit role in solvability of the system.  Another method is called the normalization method in which both the system matrix and the constant vector play explicit roles in the solution process. Each of these methods yields the maximal solution if it exists. Finally, we show the maximal solutions obtained from these methods and some previous  methods are all identical.	
	\end{abstract}
	
	{\small {\it key words}: Semiring; Idempotent semifield; Linear system of equations }\\[0.3cm]
		
	\section{Introduction}
	Solving linear systems of equations is an important aspect of many mathematical problems and their applications in various areas of
	science and engineering. Recently, idempotent semifields have been finding applications in computer science, optimization theory and control theory (see~\cite{Aceto},~\cite{McEneaney},~\cite{krivulin1}). This makes linear system solution over idempotent semifields ever so important. In this work, we intend to extend the solution methods presented in \cite{jam.cramer} and \cite{oli.norm}, namely, the pseudo-inverse method, the extended Cramer's rule, and the normalization method, to idempotent semifields.
	
	The notion of semirings was first introduced by Vandiver \cite{vandiver} in 1934 as a generalization of rings, where division and subtraction are not necessarily defined.  A semiring $(S,\oplus,\otimes,0,1)$ is an algebraic structure in which $(S, \oplus )$ is a commutative monoid with an identity element $0$ and $(S,\otimes)$ is a monoid with an identity element 1, connected by ring-like distributivity. The additive identity $0$ is multiplicatively absorbing, and $0 \neq 1$. A commutative semiring in which every nonzero element is multiplicatively invetrible is called a semifield.
	
	We obtain maximal solutions through the proposed methods for solving linear systems over idempotent semifields if they exist. A maximal solution is determined with respect to the defined total order on the idempotent semifield. Importantly, we prove that the maximal solutions obtained from the presented methods are all identical.
	
	In Section~3, we first present a necessary and sufficient condition based on the pseudo-inverse of the system matrix, whenever the determinant of the system matrix in its idempotent semifield is nonzero. We also propose the extended Cramer's rule to find the maximal solution. Moreover, by introducing the normalization method over idempotent semifields, we present an equivalent condition on column minimum elements of the associated normalized matrix of a linear system to determine the maximal solution.
	
	Section~4 concerns the comparison of the maximal solutions obtained from the presented methods in Section~3. We show that the pseudo-inverse method and the normalization method are equivalent. Finally, we prove that the column minimum elements of the associated normalized matrix of a linear system are the same as the maximal solution obtained from the $ LU $-method, which is introduced in \cite{jam.LU}.
	
	\section{Definitions and preliminaries}
	In this section, we give some definitions and preliminary notions. For convenience, we use $ \mathbb{N} $, and $\underline{n}$ to denote the set of all positive integers, and the set $\{ 1,2,\cdots,n\}$ for $n \in \mathbb{N}$, respectively.
	\begin{definition}(See \cite{golan})
		A semiring $ (S,\oplus,\otimes, 0,1)$ is an algebraic system consisting of a nonempty set $ S $ with two binary operations,  addition and multiplication, such that the following conditions hold:
		\begin{enumerate}
			\item{$ (S, \oplus) $ is a commutative monoid with identity element $ 0 $};
			\item{ $ (S, \otimes) $ is a monoid with identity element $ 1 $};
			\item{ Multiplication distributes over addition from either side, that is $ a\otimes(b\oplus c)= (a\otimes b)\oplus (a\otimes c) $ and $ (b \oplus c)\otimes a=(b\otimes a) \oplus (c\otimes a) $ for all $ a, b, c \in S $};
			\item{ The neutral element of $ S $ is an absorbing element, that is $ a\otimes 0 =0= 0 \otimes a  $ for all $ a \in S $};
			\item{ $ 1 \neq 0 $}.
		\end{enumerate}
	\end{definition}
	A semiring $ S $ is called commutative if $ a\otimes b = b \otimes a $. It is zerosumfree if $a\oplus b=0$ implies that $a=0=b$, for all $ a, b \in S $.
	\begin{definition}
		A semiring $S$ is called additively idempotent if  $a\oplus a=a $, for every $ a\in S $ .\\
	\end{definition}
	\begin{definition}
		A commutative semiring $(S,\oplus,\otimes,0,1)$ is called a semifield if every nonzero element of $S$ is multiplicatively invetrible.
	\end{definition}
	\begin{definition}
		The semifield $(S,\oplus, \otimes, 0,1)$ is idempotent if it is an additively idempotent, totally ordered, and radicable semifield. Note that the radicability implies the power $a^{q}$ be defined for any $a \in S \setminus \{0\}$ and $q \in \mathbb{Q}$ (rational numbers). In particular, for any non-negative integer $p$ we have
		\begin{center}
			$a^{0}=1 ~~~~~~~ a^{p}=a^{p-1}a~~~~~ a^{-p}=(a^{-1})^{p} $
		\end{center}
	\end{definition}
	The totally ordered operator ``$\leq_{S} $" on an idempotent semifield  defines the following partial order:
	\begin{center}
		$a \leq_{S} b\Longleftrightarrow a\oplus b = b $,
	\end{center}
	that is induced by additive idempotency. Notice that the last partial order equips addition to an extremal property in the form of the following inequalities
	$$ a \leq_{S} a \oplus b ,~~~~~~ b\leq_{S} a \oplus b ,$$
	which makes idempotent semifields zerosumfree, since we have $a \geq_{S} 0$ for any $ a  \in S$.\\ The total order defined on $ S $ is compatible with the algebraic operations, that is:
	$$ (a \leq_{S} b \wedge c \leq_{S} d) \Longrightarrow a\oplus c \leq_{S} b \oplus d, $$
	$$ (a \leq_{S} b \wedge c \leq_{S} d) \Longrightarrow a\otimes c \leq_{S} b \otimes d,$$
	for any $ a,b,c,d \in S $.
	Throughout this article, we consider the notation $``\geq_{S}" $ as the converse of the order $ ``\leq_{S}" $, which is a totally ordered operator satisfying $ a \geq_{S} b $ if and only if $ b\leq_{S} a $, for any $ a, b \in S $. Furthermore, we use  $ ``a <_{S} b" $ whenever, $ a\leq_{S} b $ and $ a \neq b$. The notation $ ``>_{S} " $ is defined similarly.
	\begin{definition}
		Let $ (S, \oplus, \otimes, 0_{S}, 1_{S} ) $ be a semifield. A semivector space over a semifield $ S $ is a commutative monoid $ (\mathcal{V},+) $ with identity element $ 0_{\mathcal{V}}$ for which we have a scalar multiplication function $S \times \mathcal{V} \longrightarrow \mathcal{V} $, denoted by $ (s, v)\mapsto sv $, which satisfies the following conditions for all $ s, s' \in S $ and $ v, v' \in \mathcal{V} $:
		\begin{enumerate}
			\item{$ (s\otimes s')m=s(s'v) $ };
			\item{$ s(v+v')= sv +sv'$};
			\item{$ (s\oplus s')v=sv +s'v $};
			\item{$ 1_{S}v=v$};
			\item{$ s0_{\mathcal{V}}=0_{\mathcal{V}}=0_{S}v$}.
		\end{enumerate}
		The elements of a semivector space $ \mathcal{V} $ are called vectors.
	\end{definition}
	Let $ \mathcal{W} $ be a nonempty subset of $ \mathcal{V} $. Then the expression $ \displaystyle{\sum_{i=1}^{n}}s_{i}\omega_{i} $ with $ \omega_{i} \in \mathcal{W} $, for any $ i \in \underline{n}$ is a linear combination of the elements of $ \mathcal{W} $. Indeed, the nonempty subset $ \mathcal{W} $ of vectors is called linearly independent if no vector of $ \mathcal{W} $ is a linear combination of other vectors of $ \mathcal{W} $. Otherwise, it is a linearly dependent subset (see \cite{taninner}).
	
	Let $S$ be an idempotent semifield. We denote the set of all $m \times n$ matrices over $S$ by $M_{m \times n}(S)$.
	For $A \in M_{m \times n} (S)$, we denote by $a_{ij}$ and $A
	^{T}$ the $(i,j)$-entry of $A$ and the transpose of $A$, respectively.\\
	For any $A, B \in M_{m \times n}(S)$, $C \in M_{n \times l}(S)$ and $\lambda \in S$, we define:
	$$A+B = (a_{ij} \oplus b_{ij})_{m \times n},
	$$ $$AC=(\bigoplus_{k=1}^{n} (a_{ik} \otimes b_{kj}))_{m \times l},$$
	and $$\lambda A=(\lambda \otimes a_{ij})_{m\times n}.$$
	It is easy to verify that $M_{n}(S):=M_{n \times n}(S)$ forms a semiring with respect to the matrix addition and the matrix multiplication.
	\begin{definition}(See \cite{wilding})
		The set of all linear combinations of columns (rows) of an arbitrary matrix $ A $, denoted by $ Col(A) ~(Row(A)) $, is the column (row) space of $ A $. The smallest $ n \in\mathbb{N} $ for which there exists a set of generator of $ Col(A) ~(Row(A)) $ with cardinality $ n $ is the column (row) rank of $ A$. The column rank and the row rank of a matrix over an idempotent semifield are not necessarily equal.
	\end{definition}

	Let $A \in M_{n}(S)$, $\mathcal{S}_{n}$ be the symmetric group of degree $n \geq 2$ , and $\mathcal{A}_{n}$ be the alternating group on $n$ such that
	$$\mathcal{A}_{n}= \{ \sigma \vert \sigma \in \mathcal{S}_{n} ~\text{and}~\sigma \text{~is~an~even~permutation}\}.$$
	The positive determinant, $\vert A \vert ^{+}$, and negative determinant, $\vert A \vert ^{-}$, of $A$ are
	$$\vert A \vert ^{+}=\bigoplus_{\sigma \in \mathcal{A}_{n}} \bigotimes_{i=1}^{n} a_{i\sigma (i)},$$
	and
	$$\vert A \vert ^{-}= \bigoplus_{\sigma \in \mathcal{S}_{n} \backslash \mathcal{A}_{n}} \bigotimes_{i=1}^{n} a_{i\sigma (i)}. $$
	Clearly, if $S$ is a commutative ring, then $\vert A \vert=\vert A \vert ^{+} - \vert A \vert ^{-}$.
	\begin{definition}
		Let $S$ be a semiring. A bijection $\varepsilon$ on $S$ is called an $\varepsilon$-function of $S$ if $\varepsilon(\varepsilon(a))=a$, $\varepsilon(a \oplus b)= \varepsilon(a) \oplus \varepsilon(b)$, and $\varepsilon(a \otimes b)=\varepsilon(a) \otimes b=a \otimes \varepsilon(b)$ for all $a, b \in S$. Consequently, $\varepsilon(a) \otimes \varepsilon(b)=a \otimes b$ and $\varepsilon(0)=0$.\\
		The identity mapping: $a \mapsto a$ is an $\varepsilon$-function of $S$ that is called the identity $\varepsilon$-function.
	\end{definition}
	\begin{remark}
		Any semiring $S$ has at least one $\varepsilon$-function since the identical mapping of $S$ is an $\varepsilon$-function of $S$. If $S$ is a ring, then the mapping : $a \mapsto -a$, $( a \in S)$ is an $\varepsilon$-function of $S$.
	\end{remark}
	\begin{definition}
		Let $S$ be a commutative semiring with an $\varepsilon$-function $\varepsilon$, and $A \in M_{n}(S)$. The $\varepsilon$-determinant of $A$, denoted by $\det_{\varepsilon}(A)$, is defined by
		$$ det_{\varepsilon}(A)= \bigoplus_{\sigma \in \mathcal{S}_{n}}\varepsilon ^{\tau(\sigma)}(a_{1\sigma(1)} \otimes a_{2\sigma(2)} \otimes \cdots \otimes a_{n\sigma(n)})$$
		where $\tau(\sigma)$ is the number of the inversions of the permutation $\sigma$, and $\varepsilon^{(k)}$ is defined by $\varepsilon^{(0)}(a)=a$ and $\varepsilon^{(k)}(a)=\varepsilon^{(k-1)}(\varepsilon(a))$ for all positive integers $k$.
		Since $\varepsilon^{(2)}(a)=a$, $\det_{\varepsilon}(A)$ can be rewritten in the form of		
		$$det_{\varepsilon}(A)=\vert A \vert ^{+} \oplus \varepsilon(\vert A \vert^{-})$$.
	\end{definition}
	Let $(S,\oplus, \otimes, 0,1)$ be an idempotent semifield, $A \in M_{m\times n}(S)$, and $b \in S^{m}$ be a column vector. Then the $i-$th equation of the linear system $AX=b$ is
	$$ \bigoplus_{j=1}^{n}( a_{ij} \otimes x_{j})= (a_{i1}\otimes x_{1})\oplus(a_{i2}\otimes x_{2}) \oplus \cdots \oplus (a_{in} \otimes x_{n}) =b_{i}.$$
	\begin{definition}
		A solution $X^{*}$ of the system $AX=b$ is called maximal if $X \leq X^{*}$ for any solution $X$.
	\end{definition}
	\begin{definition}
		A vector $b\in S^{m}$ is called  regular if it has no zero element.
	\end{definition}
	Without loss of generality, we can assume that $b$ is regular in the system $AX=b$. Otherwise, let $b_{i}=0$ for some $i \in \underline{n}$. Then in the $i-$th equation of the system, we have $a_{ij}\otimes x_{j}=0$ for any $j \in \underline{n}$, since $S$ is zerosumfree. As such, $x_{j}=0$ if $a_{ij}\neq 0$.  Consequently, the $i-$th equation can be removed from the system together with every column $A_{j}$ where $a_{ij} \neq 0$, and the corresponding $x_{j}$ can be set to $0$.\\
	\begin{definition}
		Let the linear system of equations $ AX=b $ have solutions. Suppose that $ A_{j_{1}}, A_{j_{2}}, \cdots, A_{j_{k}} $ are linearly independent columns of $ A $, and $ b $ is a linear combination of them. Then the corresponding variables, $ x_{j_{1}}, x_{j_{2}}, \cdots, x_{j_{k}} $ , are called leading variables and other variables are called free variables of the system $ AX= b$. \\
		The degrees of freedom of the linear system  $ AX=b $, denoted by $ \mathcal{D}_{f}$, is the number of free variables.
	\end{definition}
	\section{Methods for solving linear systems over idempotent semifields }
	In this section, we present methods to solve a linear system of equations over an idempotemt semifield. Those methods are pseudo-inverse method, extended Cramer's rule and normalization method, which determine the maximal solution of the system. Throughout this section, we consider $ S $ as an idempotent semifield.
	\begin{definition}
		Let $A \in M_{n}(S)$ and $\varepsilon$ be an $\varepsilon$-function on $ S $. The $\varepsilon$-adjoint matrix $A$, denoted by $adj_{\varepsilon}(A)$, is defined as follows.
		$$ adj_{\varepsilon}(A)=((\varepsilon^{(i+j)} det_{\varepsilon}(A(i | j)))_{n \times n})^{T}, $$
		where $ A(i|j) $ is the $ (n-1) \times (n-1) $ submatrix of $ A $ obtained from $ A $ by removing the $ i $-th row and the $ j $-th column.
	\end{definition}
	\begin{theorem} (See \cite{tan1})\label{r-c} Let $A \in M_{n}(S)$. We have
		\begin{enumerate}
			\item $Aadj_{\varepsilon}(A)=(\det_{\varepsilon}(A_{r}(i \Rightarrow j)))_{n \times n}$,
			\item $adj_{\varepsilon}(A)A=(\det_{\varepsilon}(A_{c}(i \Rightarrow j)))_{n \times n}$,
		\end{enumerate}
		where $A_{r}(i \Rightarrow j)$ ($A_{c}(i \Rightarrow j)$) denotes the matrix obtained from $A$  by replacing the $j$-th row (column) of $A$ by the $i$-th row (column) of $A$.
	\end{theorem}
	\begin{definition}
		Let $A \in M_{n}(S)$ and $\det_{\varepsilon}(A)$ be a nonzero element of $ S $. Then the square matrix $ \det_{\varepsilon}(A)^{-1} adj_{\varepsilon}(A)$, denoted by $ A^{-}$, is called the pseudo-inverse of $A$.
	\end{definition}
	\begin{corollary}\label{AA^{-}cor}
		Let $ A\in M_{n}(S)$. Then the elements of the multiplication matrix   $ AA^{-} $ are
		$$ (AA^{-})_{ij}= det_{\varepsilon}(A)^{-1} det_{\varepsilon}(A_{r}(i \Rightarrow j)). $$
		In particular, the diagonal entries of the matrix  $AA^{-}$ are $ 1 $. Furthermore, the entries of the matrix $ A^{-}A $ are defined analogusly.
	\end{corollary}
	\begin{proof}
		\begin{align*}
		(AA^{-})_{ii} &\ = det_{\varepsilon}(A)^{-1}(Aadj_{\varepsilon}(A))_{ii}\\
		&\ = det_{\varepsilon}(A)^{-1} det_{\varepsilon}(A_{r}(i \Rightarrow i))\\
		&\ = det_{\varepsilon}(A)^{-1} det_{\varepsilon}(A)\\
		&\ =  1
		\end{align*}
	\end{proof}
	\paragraph{\textbf{Pseudo-inverse method.}} First we introduce the pseudo-inverse method to obtain a maximal solution of the linear system. In the following theorem, we present a necessary and sufficient condition on the system matrix over idempotent semifields which is an extension of Theorem~3 in \cite{jam.cramer} for ``$ \max-plus$ algebra ".
	\begin{theorem}\label{psuedothm}
		Let $ AX=b $, where $ A \in M_{n}(S) $, and $ b $ be a regular vector of size $ n $. Then the system $ AX=b $ has the maximal solution $X^{*}=A^{-}b $ with $X^{*}=(x_{i}^{*})_{i=1}^{n}$ if and only if $(AA^{-})_{ij}\otimes b_{j} \leq_{S} b_{i} $ for any $i,j \in \underline{n}$.
	\end{theorem}	
	\begin{proof}
		Assume that $X^{*}=A^{-}b $ is a maximal solution of the system $ AX=b $. Then $ AA^{-}b=b $, that is in the $i$-th equation of $ AA^{-}b=b $ for any $i \in \underline{n}$ we have
		$$ ((AA^{-})_{i1} \otimes b_{1}) \oplus ((AA^{-})_{i2} \otimes b_{2})\oplus ((AA^{-})_{in} \otimes b_{n})= b_{i}. $$
		The induced partial order on the idempotent semifield $ S $ implies that $(AA^{-})_{ij}\otimes b_{j} \leq_{S} b_{i} $ for any $ j \in \underline{n}$.\\
		Conversely, suppose that $(AA^{-})_{ij}\otimes b_{j} \leq_{S} b_{i} $ for any $i,j \in \underline{n}$. We prove that $X^{*}=A^{-}b $ is a solution of the system and it is maximal corresponding to the total order $ \leq_{S} $. Clearly, the $ i $-th component of the vector $ AX^{*} $ is
		\begin{align}
		(AX^{*})_{i} \nonumber &= (AA^{-}b)_{i}\\
		&=\bigoplus_{j=1}^{n}((AA^{-})_{ij} \otimes b_{j})\\
		&=\bigoplus_{\begin{subarray}{c}
			j=1 \\ j \neq i
			\end{subarray}}^{n} ((AA^{-})_{ij} \otimes b_{j}) \oplus ((AA^{-})_{ii} \otimes b_{i})\\
		&=b_{i}. \nonumber
		\end{align}
		The equalities $(3.1)$ and $(3.2)$ are obtained from Corollary~\ref{AA^{-}cor} and the assumption $(AA^{-})_{ij}\otimes b_{j} \leq_{S} b_{i} $ for any $ i,j \in \underline{n}$, respectively. As such, $ X^{*} $ is a solution of the system $ AX=b $.\\
		Now, we show that $ X^{*}=A^{-}b $ is the maximal solution of $ AX=b $. Since $ AX^{*}=b $, the $ i $-th equation of the system $A^{-}AX^{*}=X^{*}$ in the form
		$$ ((A^{-}A)_{i1} \otimes x^{*}_{1}) \oplus \cdots \oplus x^{*}_{i} \oplus \cdots \oplus ((A^{-}A)_{in} \otimes x^{*}_{n})= x^{*}_{i}  $$
		yields the inequalities
		\begin{equation}\label{ineq}
		(A^{-}A)_{ij} \otimes x^{*}_{j} \leq_{S} x^{*}_{i}
		\end{equation}
		for any $j \neq i$. We now consider another solution of the system such as $ Y=(y_{i})_{i}^{n}$, that is $AY=b$ and therefore $ A^{-}AY=X^{*} $. There is no loss of generality to assume that $ y_{i} \neq x^{*}_{i} $ for some $ i \in \underline{n} $ and $ y_{j} = x^{*}_{j} $ for any $ j \neq i $. Note that the $ i $-th equation of the system $ A^{-}AY=X^{*}  $ is
		$$ ((A^{-}A)_{i1} \otimes x^{*}_{1}) \oplus \cdots \oplus (A^{-}A)_{ii} \otimes y_{i} \oplus \cdots \oplus ((A^{-}A)_{in} \otimes x^{*}_{n})= x^{*}_{i}. $$
		Due to the induced partially ordered operator $`` \leq_{S}" $	over idempotent semifields, we have $ (A^{-}A)_{ii} \otimes y_{i} \leq_{S} x^{*}_{i} $ which implies $ y_{i} <_{S} x^{*}_{i} $, since $ (A^{-}A)_{ii}=1 $ and $ y_{i} \neq x^{*}_{i} $. Moreover, $ Y $ is a solution of the system $ AX=b $, so the inequalities~\ref{ineq} cannot be all proper. If all of them are proper, then   		
		$$ ((A^{-}A)_{i1} \otimes x^{*}_{1}) \oplus \cdots \oplus (A^{-}A)_{ii} \otimes y_{i} \oplus \cdots \oplus ((A^{-}A)_{in} \otimes x^{*}_{n}) <_{S} x^{*}_{i}, $$
		which is a contradiction. In fact, if all of the inequalities~(\ref{ineq}) are proper, then $ X^{*} $ is the unique solution. Otherwise, suppose that $ (A^{-}A)_{ij} \otimes x^{*}_{j}= x^{*}_{i} $ for some $ j \neq i $. Then $ Y $ is a solution of the system that satisfies $ Y \leq_{S} X^{*} $. Hence $ X^{*} $	is a maximal solution.
	\end{proof}
	In the following example, we use the necessary and sufficient condition $(AA^{-})_{ij}\otimes b_{j} \leq_{S} b_{i} $ for the system $ AX=b $ to have the maximal solution $ X^{*}=A^{-}b $.
	\begin{ex}\label{AA^{-}ex}
		Let $ A\in M_{4}(S) $, where $ S=\mathbb{R}_{\max,\times}=(\mathbb{R_{+}} \cup \{0\}, \max,\times,0,1) $, and $\mathbb{R_{+}} $ be the set of all positive real numbers. Then the defined total order on ``$\max-\rm times$ algebra" is the standard less than or equal relation $ `` \leq " $ over $\mathbb{R}$ and the multiplication  $ a \times b^{-1} $, denoted by  $ \dfrac{a}{b}$, is the usual real numbers division. Consider the following system $ AX=b $:	
		\[
		\left[
		\begin{array}{cccc}
		5&7&9&10\\
		4&2&0&7\\
		3&0&3&5\\
		1&8&1&6
		\end{array}
		\right]
		\left[
		\begin{array}{c}
		x_{1}\\
		x_{2}\\
		x_{3}\\
		x_{4}
		\end{array}
		\right]
		=
		\left[
		\begin{array}{c}
		27\\
		16\\
		12\\
		24
		\end{array}
		\right],
		\]
		where $ \det_{\varepsilon}(A)=a_{13} \times a_{24} \times a_{31} \times a_{42} = 1512 $. Due to Corollary~\ref{AA^{-}cor}, we have $ (AA^{-})_{ij}=\dfrac{\det_{\varepsilon}(A_{r}(i \Rightarrow j)) }{\det_{\varepsilon}(A)} $, for any $i,j \in \{ 1,\cdots ,4 \} $. As such, $ AA^{-} $ is
		\[
		AA^{-}=\left[
		\begin{array}{cccc}
		1&\frac{10}{7}&\frac{40}{21}&\frac{7}{8}\\
		\frac{4}{9}&1&\frac{4}{3}&\frac{7}{18}\\
		\frac{1}{3}&\frac{5}{7}&1&\frac{7}{24}\\
		\frac{8}{21}&\frac{6}{7}&\frac{8}{7}&1\\
		\end{array}
		\right].
		\]
		By Theorem~\ref{psuedothm}, we must check the condition $(AA^{-})_{ij} \times b_{j} \leq b_{i} $, for any $i,j \in \{ 1,\cdots ,4 \} $. In order to simplify the computation procedure, we check $ (AA^{-})_{ij} \leq \dfrac{b_{i}}{b_{j}} \leq \dfrac{1}{(AA^{-})_{ji}}  $, for any $ 1 \leq i\leq j \leq 4 $.\\
		Since these inequalities hold, for instance $ (AA^{-})_{12} \leq \dfrac{27}{16} \leq \dfrac{1}{(AA^{-})_{21}} $, the system $ AX=b $ has the maximal solution $ X^{*}=A^{-}b$:
		\[
		X^{*}=\left[
		\begin{array}{cccc}
		\frac{1}{9}&\frac{5}{21}&\frac{1}{3}&\frac{7}{72}\\
		\frac{1}{21}&\frac{3}{28}&\frac{1}{7}&\frac{1}{8}\\
		\frac{1}{9}&\frac{10}{63}&\frac{40}{189}&\frac{7}{72}\\
		\frac{4}{63}&\frac{1}{7}&\frac{4}{21}&\frac{1}{18}
		\end{array}
		\right]
		\left[
		\begin{array}{c}
		27\\
		16\\
		12\\
		24
		\end{array}
		\right]
		=
		\left[
		\begin{array}{c}
		4\\
		3\\
		3\\
		\frac{16}{7}
		\end{array}
		\right].
		\]
	\end{ex}
	\paragraph{\textbf{Extended Cramer's rule.}} We know that in classic linear algebra over fields, commutative rings and in a fortiori over commutative semirings (see \cite{taninverse}), Cramer's rule is a useful method for determining the unique solution of a linear system whenever the system matrix is invertible. Furthermore, Sararnrakskul proves that a square matrix $A$ over a semifield is invertible if and only if every row and every column of $A$ contains exactly one nonzero element ( see \cite{sara}). As such, Cramer's rule can be used only for this type of matrices over an idempotent semifield. Now, we introduce a new version of Cramer's rule to obtain the maximal solution of a linear system, based on the pseudo-inverse of the system matrix. In the next theorem, we present the extended Cramer's rule over idempotent semifields whenever the $ \varepsilon $-determinant of the system matrix is nonzero, which is an extension of Theorem~4 in \cite{jam.cramer}.
	\begin{theorem}\label{cramerthm}
		Let $ A \in M_{n}(S) $, $ b $ be a regular vector of size $ n $, and $ \det_{\varepsilon}(A) \neq 0$. Then the following statements are equivalent:
		\begin{enumerate}
			\item  the system $ AX= b $ has the maximal solution $ X^{*}=(\det_{\varepsilon}(A)^{-1} \otimes \det_{\varepsilon}(A_{[i]}))_{i=1}^{n}$, and
			\item  $(AA^{-})_{ij}\otimes b_{j} \leq_{S} b_{i} $ for any $i,j \in \underline{n}$,
		\end{enumerate}
		where $A_{[i]} $ is the matrix obtained from replacing the $ i $-th column of $ A $ by the column vector $b $.	
	\end{theorem}
	\begin{proof}
		Due to Theorem~\ref{psuedothm}, the inequalities $ (AA^{-})_{ij}\otimes b_{j} \leq_{S} b_{i} $ for every $i,j \in \underline{n}$  are necessary and sufficient conditions for the system $ AX=b $ to have the maximal solution $ X^{*}=A^{-}b$. Consequently,
		
		\begin{align}
		x^{*}_{i} \nonumber &= (A^{-}b)_{i} \\ \nonumber
		&=det_{\varepsilon}(A)^{-1} \otimes (adj_{\varepsilon}(A)b)_{i} \\ \nonumber
		&=det_{\varepsilon}(A)^{-1} \otimes ( \bigoplus_{j=1}^{n}((adj_{\varepsilon}(A))_{ij} \otimes b_{j} ) ) \\ \nonumber
		&=det_{\varepsilon}(A)^{-1} \otimes ( \bigoplus_{j=1}^{n}(det_{\varepsilon}(A(j|i)) \otimes b_{j} ) )\\
		&=det_{\varepsilon}(A)^{-1} det_{\varepsilon}(\left[
		\begin{array}{ccccccc}
		a_{11}&\cdots&a_{1(j-1)}&b_{1}&a_{1(j+1)}\cdots&a_{1n}\\
		\vdots&~&\vdots&\vdots&\vdots&\vdots&~\\
		a_{n1}&\cdots&a_{n(j-1)}&b_{n}&a_{n(j+1)}\cdots&a_{nn}
		\end{array}
		\right] ) \\
		&=det_{\varepsilon}(A)^{-1} \otimes det_{\varepsilon}(A_{[i]}), \nonumber
		\end{align}
		where the equality $ (3.4) $ is obtained from Laplace's theorem for semirings (see Theorem~3.3 in~\cite{tan1}).
	\end{proof}	
	The normalization method for solving a system of linear equations over ``$\max-\rm plus$ algebra" is already introduced in \cite{oli.norm}. Here, we extend this method to idempotent semifields. To this end, we start by defining the following normalization method.
	
	\paragraph{\textbf{Normalization method.}} Consider the linear system $ AX=b $, where $ A \in M_{m \times n}(S) $, $ b $ is a regular vector of size $ m $ and $ X $ is an unknown vector of size $ n $. Let the $ j$-th column of $ A $, denoted by $ A_{j}$, be a regular vector of size $ m $, for any $ j \in \underline{n}$. In fact, the system matrix, $ A $, contains no zero element. Then the normalized matrix of $ A $ is
	\[
	\tilde{A}=
	\left[
	\begin{array}{c|c|c|c}
	\hat{A}_{1}^{-1} A_{1} & \hat{A}_{2}^{-1} A_{2} &\cdots& \hat{A}_{n}^{-1} A_{n}
	\end{array}
	\right],
	\]
	where $ \hat{A}_{j}= \sqrt[m]{a_{1j}\otimes\cdots\otimes a_{mj}} $, for any $ j \in \underline{n}$. The normalized vector of the regular vector $ b $ is defined similarly as follows.
	$$\tilde{b} = \hat{b}^{-1} b ,$$
	where  $ \hat{b}= \sqrt[m]{b_{1}\otimes\cdots\otimes b_{m}}.$
	As such, the normalized system corresponding to the system $ AX=b $, denoted by $ \tilde{A}Y=\tilde{b} $, is obtained as follows.
	\begin{align*}
	AX=b
	&\ \Rightarrow \bigoplus_{j=1}^{n}( A_{j}x_{j} ) =b \\
	&\ \Rightarrow \bigoplus_{j=1}^{n}( A_{j}(\hat{A}_{j}^{-1} \otimes \hat{A}_{j} \otimes x_{j}) ) =(\hat{b} \otimes \hat{b}^{-1})b \\
	&\ \Rightarrow \bigoplus_{j=1}^{n}( \tilde{A}_{j}( \hat{A}_{j}\otimes x_{j}) ) =\hat{b}\tilde{b} \\
	&\ \Rightarrow \bigoplus_{j=1}^{n}( \tilde{A}_{j}( \hat{A}_{j} \otimes \hat{b}^{-1}\otimes x_{j}) ) =\tilde{b} \\
	&\ \Rightarrow \bigoplus_{j=1}^{n}(\tilde{A}_{j} y_{j} ) =\tilde{b} \\
	&\ \Rightarrow \tilde{A}Y=\tilde{b},
	\end{align*}
	where $ Y= (\hat{A}_{j}\otimes \hat{b}^{-1} )X $. The $ i$-th equation of the normalized system $ \tilde{A}Y=\tilde{b} $ is
	$$ (\tilde{a}_{i1} \otimes y_{1}) \oplus \cdots \oplus(\tilde{a}_{in} \otimes y_{n})=\tilde{b}_{i}, $$
	which implies $ y_{j} \leq_{S} \tilde{b}_{i}$, for any $ i \in \underline{m}$ and  $ j \in \underline{n} $. The associated normalized matrix of the system $ AX=b $ can be defined as $ Q= (q_{ij}) \in M_{m\otimes n}(S) $ with $ q_{ij}= \tilde{b}_{i} \otimes \tilde{a}_{ij}^{-1} $. We choose $ y_{j} $ as the minimum element of the $ j$-th column of $ Q $ with respect to the order $ `` \leq_{S} " $. Note further that in the normalization process of  column $ A_{j} $, we disregard zero elements, that is if $ a_{ij}=0 $ for some $ i \in \underline{m} $ and $ j \in \underline{n} $, then
	$$ \hat{A}_{j}= \sqrt[m]{a_{1j}\otimes\cdots\otimes a_{(i-1)j} \otimes a_{(i+1)j}\otimes\cdots\otimes a_{mj}}. $$
		and $ \tilde{a}_{ij}= 0 \otimes \hat{A}_{j}^{-1} = 0 $. As such, we set  $ q_{ij} := 0^{-} $, where  $ a <_{S} 0^{-}$ for any $ a \in S $. The $ j$-th column minimum element of $ Q $ is therefore determined regardless of $q_{ij} $. Thus, without loss of generality, we consider every column of the system matrix to be regular.
	
	In the next theorem, we present a necessary and sufficient condition on the column minimum elements of the associated normalized matrix to solve the systems $ \tilde{A}Y=\tilde{b} $ and consequently $ AX=b$, since we choose $ y_{j} $ as the $ j$-th column minimum element, for any $ j \in \underline{n} $.
	\begin{theorem}\label{normthm}
		Let $ AX=b $ be a linear system of equations with $ A \in M_{m \times n }(S) $ and the regular $ m$-vector $ b $. Then the system $ AX=b $ has solutions if and only if every row of the associated normalized matrix $ Q $ contains at least one column minimum element.
	\end{theorem}
	\begin{proof}
		Let the system $ AX=b $ have solutions. Assume that the $ i$-th row of $ Q $ contains no column minimum element, that is $ y_{j} \neq q_{ij} $ and so $ y_{j} <_{S} \tilde{b}_{i} \otimes \tilde{a}_{ij}^{-1} $, for any $ j \in \underline{n} $. Thus, in the $ i$-th equation of the system $ \tilde{A}Y=\tilde{b} $ we have
		$$ \bigoplus_{j=1}^{n}(\tilde{a}_{ij} \otimes y_{j}) <_{S} \tilde{b}_{i}, $$
		which implies the system $\tilde{A}Y=\tilde{b} $ and consequently the system $ AX=b$ have no solution and is a contradiction.\\
		Conversely, suppose that every row of $ Q $ contains at least one column minimum element, that is for every $ i \in \underline{m} $, there exist some $ k \in \underline{n} $ such that  $ y_{k} = \tilde{b}_{i} \otimes \tilde{a}_{ik}^{-1} $. Then the $ i$-th equation of the system $ \tilde{A}Y=\tilde{b} $, for every $ i \in \underline{m} $ is as follows
		\begin{align}
		\bigoplus_{j=1}^{n}(\tilde{a}_{ij} \otimes y_{j}) \nonumber &= \bigoplus_{\begin{subarray}{c}
			j=1 \\ j \neq k
			\end{subarray}}^{n}(\tilde{a}_{ij} \otimes y_{j}) \oplus (\tilde{a}_{ik} \otimes y_{k}) \\
		&=\bigoplus_{\begin{subarray}{c}
			j=1 \\ j \neq k
			\end{subarray}}^{n}(\tilde{a}_{ij} \otimes y_{j}) \oplus \tilde{b}_{i} \\ \nonumber
		&=\tilde{b}_{i}, \nonumber
		\end{align}
		where the equality $ (3.5) $ is obtained from choosing $ y_{j} $ as the $ j$-th column minimum element of the matrix $ Q $. 
		That means the system $ \tilde{A}Y=\tilde{b} $ and consequently, the system $ AX=b $ has solutions. Indeed, the obtained solution $ X=((\hat{A}_{i} \otimes \hat{b}_{i}^{-1})\otimes y_{i})_{i=1}^{n} $ of the system $ AX=b $ is maximal.
	\end{proof}
	\begin{remark}
		The normalization method and the associated normalized matrix of a linear system provide comprehensive information which enables us to compute the degrees of freedom and determine the column rank and the row rank of a matrix. These applications over idempotent semifields are similar to the descriptive methods stated in \cite{oli.norm} over $``\max-plus~algebra"$.
	\end{remark}
	\paragraph{\textbf{Equivalence classes of matrices.}} Let $ A, A' \in M_{m\times n}(S) $. We say $ A$ is equivalent to $ A'$ if the $ j$-th column of $ A' $ is a scalar multiple of the $ j$-th column of $ A $, for any $ j \in \underline{n} $, that is
	$$A \sim A' \Longleftrightarrow  A'= \left[
	\begin{array}{c|c|c}
	\alpha_{1}A_{1}&\cdots&\alpha_{n}A_{n}
	\end{array}
	\right],
	$$
	for some $ \alpha_{1}, \cdots, \alpha_{n} \in S \setminus\{ 0  \} $. As such, the equivalence class of $ A $ is
	$$\left[ A \right]= \{A'\in M_{m \times n}(S) \vert A \sim A'\}. $$
	The equivalence relation on vectors is defined analogously.
	\section{Equivalance of the solution methods}
	In this section, we show that the presented solution methods in the previous section are equivalent and their maximal solutions are identical.
	
	In the next theorem, we intend to show the equivalence of the pseudo-inverse method and the normalization method for solving a linear system.
	\begin{theorem}\label{A^{-}norm}
		Let $ AX= b $, where $ A \in M_{n}(S), $ and $ b $ be a regular vector of size $ n $. If $ (AA^{-})_{ij} \otimes b_{j} \leq_{S} b_{i} $, for any $i,j \in \underline{n} $, then $ (A^{-}b)_{k} = b_{i} \otimes a_{ik}^{-1} $, for some $ i \in \underline{n} $ and for any $ k \in \underline{n} $.
	\end{theorem}
	\begin{proof}
		By Theorem~\ref{psuedothm}, the inequalities $ (AA^{-})_{ij} \otimes b_{j} \leq_{S} b_{i} $ implies that the system $ AX=b $ have solutions. Due to Theorem~\ref{normthm}, every row of the associated normalized matrix $ Q $ contains at least one column minimum element. As such, the maximal solution of the system  $ AX=b $ through the normalization method is $ X^{\bullet}= (x^{\bullet}_{k})_{k=1}^{n}$, where
		$$ x^{\bullet}_{k}= \tilde{b}_{i} \otimes \tilde{a}_{ik}^{-1} \otimes \hat{b} \otimes \hat{A}_{k}^{-1}= b_{i} \otimes a_{ik}^{-1}, $$
		for some $ i \in \underline{n}$. Let $ A^{-}= (a^{-}_{ij}) $, then the inequalities $ (AA^{-})_{ij} \otimes b_{j} \leq_{S} b_{i} $ implies that
		\begin{align}
		&\ \nonumber \bigoplus_{k=1}^{n} (a_{ik}\otimes a^{-}_{kj}) \otimes b_{j} \leq_{S} b_{i},\qquad~~~~ \forall i,j \in \underline{n} \\ \nonumber
		&\ \Rightarrow  \bigoplus_{k=1}^{n} (a_{ik}\otimes a^{-}_{kj} \otimes b_{j}) \leq_{S} b_{i},\qquad \forall i,j \in \underline{n} \\ \nonumber
		&\ \Rightarrow  a^{-}_{kj} \otimes b_{j} \leq_{S} b_{i} \otimes a_{ik}^{-1},\qquad~~~~~~~ \forall i,j,k \in \underline{n} \\ \nonumber
		&\ \Rightarrow  \bigoplus_{j=1}^{n} (a^{-}_{kj} \otimes b_{j}) \leq_{S} b_{i} \otimes a_{ik}^{-1},\qquad  \forall i,k \in \underline{n} \\
		&\ \Rightarrow (A^{-}b)_{k} \leq_{S} b_{i} \otimes a_{ik}^{-1},\qquad~~~~~~~~ \forall i,k \in \underline{n}.
		\end{align}
		On the other hand, $X^{*}=A^{-}b $ is the maximal solution of the system $ AX=b $. As such, $ x^{\bullet}_{k}= b_{i} \otimes a_{ik}^{-1} \leq_{S} (A^{-}b)_{k} $, for any $ k \in \underline{n}$. This inequality and (4.1) lead to $ (A^{-}b)_{k} = b_{i} \otimes a_{ik}^{-1} $, for some $ i \in \underline{n} $ and any $ k \in \underline{n} $. Consequently, these two methods have the same maximal solution, $X^{\bullet} =X^{*} $
	\end{proof}
	\begin{ex}
		Consider the linear system $ AX=b $ given in Example~\ref{AA^{-}ex}. We want to solve the system through the normalization method. To this end, we must construct the associated normalized matrix $ Q=(\tilde{b}_{i} \otimes \tilde{a}_{ij}^{-1}) $ of the system $ AX=b $. Since the column minimum elements of $ Q $ form the maximal solution of the normalized system $ \tilde{A}Y=\tilde{b} $ and the maximal solution of the system $ AX=b $ is $ X=(x_{i})_{i=1}^{4} $, where $ x_{i}=(\hat{b}\otimes \hat{A}_{i}^{-1}) \otimes y_{i} $, for any $ i \in \{1, \cdots, 4 \} $. Without loss of generality, we can consider the matrix $ Q'= ((\tilde{b}_{i} \otimes \tilde{a}_{ij}^{-1}) \otimes (\hat{b}\otimes \hat{A}_{i}^{-1}))=(b_{i} \otimes a_{ij}^{-1}) $ which is equivalent to $ Q $ with coefficients $ \hat{b}\otimes \hat{A}_{i}^{-1} \in S\setminus \{ 0 \} $, for any $ i \in \{1, \cdots, 4 \} $. It suffices to check the condition of Theorem~\ref{normthm} on the column minimum elements of $ Q'$:
		\[
		Q'=\left[
		\begin{array}{cccc}
		\frac{27}{5}&\frac{27}{7}&\boxed{3}&\frac{27}{10}\\
		\boxed{4}&8&0^{-}&\boxed{\frac{16}{7}}\\
		\boxed{4}&0^{-}&4&\frac{12}{5}\\
		24&\boxed{3}&24&4
		\end{array}
		\right].
		\]
		Clearly, every row of $ Q' $ contains at least one column minimum element. Now, the boxed entries of $ Q' $ form the maximal solution of the system $AX=b $ as
		\[
		X^{*}=\left[
		\begin{array}{c}
		4 \\
		3 \\
		3 \\
		\frac{16}{7}
		\end{array}
		\right],
		\]
		which is the same as the maximal solution of pseudo-inverse method.
	\end{ex}
	\begin{remark}
		Theorems~\ref{A^{-}norm} and \ref{cramerthm} imply that the maximal solutions obtained from the normalization method and the extended Cramer's rule should be the same.
	\end{remark}
	The notion of generalized $LU$-factorization and the method for solving linear systems through $LU$-factorization over idempotent semifields are presented in \cite{jam.LU}. The following theorem proves that the maximal solutions of $LU$-method and the normalization method are the same.
	\begin{theorem}\label{LUnorm}
		Let $ A \in M_{n}(S)$, $ A $ have generalized $LU$-factorization and $ b $ be a regular $ n$-vector. If $ a_{ik} \otimes a_{kk}^{-1} \leq _{S} b_{i} \otimes b_{k}^{-1} $ and $ a_{(n-j)l} \otimes a_{ll}^{-1} \leq _{S} b_{(n-j)} \otimes b_{l}^{-1} $ for every  $2\leq i\leq n$, $1\leq k\leq i-1$, $1\leq j\leq n-1$, and $n-j+1\leq l\leq n$, then the $ k $-th column minimum element of the associated normalized matrix $ Q $ is $ \tilde{b}_{k} \otimes \tilde{a}_{kk}^{-1} $.
	\end{theorem}
	\begin{proof}
		Due to the extended version of  Theorem~8 in \cite{jam.LU}, the assumed inequalities imply that the linear system $ AX=b $ have the maximal solution $ X^{*}=(b_{k}\otimes a_{kk}^{-1})_{k=1}^{n} $. By Theorem~\ref{normthm}, every row of $ Q $ contains at least one column minimum element. The assumption  $ a_{ik} \otimes a_{kk}^{-1} \leq _{S} b_{i} \otimes b_{k}^{-1} $, for any  $ 2\leq i\leq n $ and $ 1\leq k\leq i-1 $ implies that
		$ b_{k} \otimes a_{kk}^{-1} \leq_{S} b_{i} \otimes a_{ik}^{-1}$ and therefore $$ (\hat{b}^{-1}\otimes b_{k}) \otimes(\hat{A}_{k} \otimes a_{kk}^{-1}) \leq_{S} (\hat{b}^{-1}\otimes b_{i}) \otimes (\hat{A}_{k} \otimes a_{ik}^{-1}), $$
		which means
		\begin{equation}\label{colmin}
		\tilde{b}_{k} \otimes \tilde{a}_{kk}^{-1} \leq_{S} \tilde{b}_{i} \otimes \tilde{a}_{ik}^{-1},
		\end{equation}
		for any  $ 2\leq i\leq n $ and $ k < i $, since the total order $ `` \leq_{S} " $ is compatible with multiplication.
		
		It suffices to show that the inequality (\ref{colmin}) holds for any $ k > i $. Letting $ i:= n-j $ and $ k:= l $, we can rewrite the inequalities  $ a_{(n-j)l} \otimes a_{ll}^{-1} \leq _{S} b_{(n-j)} \otimes b_{l}^{-1} $ as the inequalities $ a_{ik} \otimes a_{kk}^{-1} \leq _{S} b_{i} \otimes b_{k}^{-1} $, for any $ 1 \leq i \leq n-1 $ and $ k > i $. Now, the compatibility with multiplication leads to
		$$ b_{k} \otimes a_{kk}^{-1} \leq_{S} b_{i} \otimes a_{ik}^{-1} $$ and  $$ \tilde{b}_{k} \otimes \tilde{a}_{kk}^{-1} \leq_{S} \tilde{b}_{i} \otimes \tilde{a}_{ik}^{-1}. $$
		Hence, $ \tilde{b}_{k} \otimes \tilde{a}_{kk}^{-1} $ is the $ k$-th column minimum element of $ Q $.
	\end{proof}
	\begin{remark}
		Note that the maximal solution of the linear system $ AX=b $ obtained from the normalization method in Theorem~\ref{LUnorm} is
		$$ X= ( \tilde{b}_{k} \otimes \tilde{a}_{kk}^{-1} \otimes \hat{b} \otimes \hat{A}_{k}^{-1} )_{k=1}^{n}=(b_{k}\otimes a_{kk}^{-1})_{k=1}^{n}, $$
		which is the same as the maximal solution obtained from the $LU$-method.
	\end{remark}
	\begin{ex}
		Let $ A\in M_{4}(S) $ where $S=\mathbb{R}_{\min, \times}=(\mathbb{R_{+}} \cup \{+\infty\}, \min,\times,+\infty,1)$ and $\mathbb{R_{+}} $ be the set of all positive real numbers. Consider the following system $ AX=b $:	
		\[
		\left[
		\begin{array}{cccc}
		1&6&9&8\\
		6&2&7&5\\
		9&7&1&7\\
		8&5&6&3
		\end{array}
		\right]
		\left[
		\begin{array}{c}
		x_{1}\\
		x_{2}\\
		x_{3}\\
		x_{4}
		\end{array}
		\right]
		=
		\left[
		\begin{array}{c}
		4\\
		6\\
		1\\
		6
		\end{array}
		\right],
		\]
		which is given in Example~6 of \cite{jam.LU}. The system has already been solved through the $LU$-method and its maximal solution is
		\[
		X^{*}= \left[
		\begin{array}{c}
		4\\
		3\\
		1\\
		2
		\end{array}
		\right].
		\]
		Now, we solve the system by the normalization method. Without loss of generality, we consider the matrix $ Q'= (b_{i} \otimes a_{ij}^{-1} ) $, which is equivalent to the associated normalized matrix of the system $ AX=b $. Its column minimum elements determine the maximal solution of the system:
		\[
		\left[
		\begin{array}{cccc}
		\boxed{4}&\frac{2}{3}&\frac{4}{9}&\frac{1}{2}\\
		1&\boxed{3}&\frac{6}{7}&\frac{6}{5}\\
		\frac{1}{9}&\frac{1}{7}&\boxed{1}&\frac{1}{7}\\
		\frac{3}{4}&\frac{6}{5}&\boxed{1}&\boxed{2}
		\end{array}
		\right],
		\]
		where the column minimum elements of $ Q' $ are boxed with respect to the total order $`` \leq_{S}" $ on ``$\min-\rm times$ algebra" as the standard greater than or equal relation $ `` \geq " $ over $\mathbb{R}$. Thus, these elements form the maximal solution of the system, which is the same as the solution obtained from the $ LU $-method.
	\end{ex}
	\section{Concluding Remarks}\label{remarks}
	In this paper, we extended the solution methods such as the pseudo-inverse method, the extended Cramer's rule and the normalization method to idempotent semifields. Applying each of these methods, we determined the maximal solution of a linear system. Importantly, we proved that the maximal solution obtained from the $ LU $-method and the above-mentioned methods are identical.
		
\end{document}